\documentclass[12pt,amscd]{amsart}
\usepackage[all]{xy}
\usepackage{graphicx}
\usepackage{amsmath,amsxtra,amssymb,latexsym, amscd,amsthm}
\usepackage{indentfirst}
\usepackage[mathscr]{eucal}
\newtheorem{thm}{Theorem}[section]
\newtheorem{cor}[thm]{Corollary}
\newtheorem{lem}[thm]{Lemma}
\newtheorem{prop}[thm]{Proposition}
\theoremstyle{definition}
\newtheorem{defn}[thm]{Definition}
\newtheorem{exm}[thm]{Example}

\numberwithin{equation}{section}

\DeclareMathOperator{\N}{\mathbb {N}}
\DeclareMathOperator{\Z}{\mathbb {Z}}
\DeclareMathOperator{\Q}{\mathbb {Q}}
\DeclareMathOperator{\R}{\mathbb {R}}
\DeclareMathOperator{\depth}{depth}
\DeclareMathOperator{\ass}{Ass}

\DeclareMathOperator{\dstab}{dstab}
\DeclareMathOperator{\astab}{astab}
\DeclareMathOperator{\conv}{conv}

\def\alb {\mathbf {\alpha}}
\def\btb {\mathbf {\beta}}

\def\zv {\mathbf 0}
\def\e {\mathbf e}
\def\f {\mathbf f}

\def\v {\mathbf v}
\def\a {\mathbf a}
\def\b {\mathbf b}
\def\x {\mathbf x}

\def\mi {\mathfrak m}
\def\nn {\mathfrak n}
\def\p {\mathfrak p}
\def\q {\mathfrak q}
\def\al {\alpha}

\def\F {\mathcal{F}}

\begin{document}

\title[Stability of Associated Primes and Depth of Integral Closures of Powers of Edge Ideals] {Stability of Associated Primes and Depth of Integral Closures of Powers of Edge Ideals}

\author[D.H. Mau]{ Dong Huu Mau}
\address{Faculty of Natural Sciences, Hanoi Metropolitan University, Ha Noi, Viet Nam.}
\email{dhmau@hnmu.edu.vn}

\author[T.N. Trung]{ Tran Nam Trung}
\address{Institute of Mathematics, VAST, 18 Hoang Quoc Viet, Hanoi, Viet Nam, and Institute of Mathematics and 
TIMAS, Thang Long University, Ha Noi, Vietnam.}
\email{tntrung@math.ac.vn}

\date{}
\dedicatory{}
\commby{}
\begin{abstract} In this paper, we study associated primes and depth of integral closures of powers of edge ideals. We provide sharp bounds on how big of powers for which the set of associated primes and the depth of integral closures of powers of edge ideals are stable.
\end{abstract}

\maketitle
\section*{Introduction}

Let $R = K[x_1,\ldots, x_r]$ be a polynomial ring over a field $K$ and $I$ an  ideal in $R$. The integral closure of $I$ is the ideal $\overline{I}$ of all elements of $R$ that satisfy an equation of the form
$$x^n +a_1x^n+\cdots+a_{n-1}x+a_n = 0, \text{ where } n\geqslant 1, a_i \in I^i, \text{ for } i =1,\ldots,n.$$

It is very hard to compute the integral closure $\overline{I}$ from $I$ except for monomial ideals (see e.g. \cite{SH}). However, the integral closures of large powers  of an ideal have very nice behavior. Ratliff \cite{R1,R2} showed that the sequence $\{\ass(R/\overline{I^n})\}_{n\geqslant 1}$ is increasing and it stabilizes for large $n$, i.e. there exists $n_0\geqslant 1$ such that
$$\ass(R/\overline{I^n}) = \ass(R/\overline{I^{n_0}}).$$

Although McAdam \cite{MA1} gave a characterization of the stable set $\ass(R/\overline{I^{n_0}})$. It is still mysterious about the set $\ass(R/\overline{I^{n}})$ for each $n$. 

If we would like to know about how big of $n_0$, according to Herzog, Rauf and Vladiou \cite{HRV} we define {\it the index of stability of integral closures of powers} of $I$ by
$$\overline{\astab}(I) = \min\{n_0 \mid \ass (R/\overline{I^n}) = \ass (R/\overline{I^{n_0}}) \text{ for all } n\geqslant n_0\}.$$

Then, the following problem gives arise.

\medskip

\noindent {\bf Problem 1}. Find a sharp bound for $\overline{\astab}(I)$.

\medskip

Another closely related sequence that we would be interested in is $\{\depth(R/\overline{I^n})\}_{n\geqslant 1}$. From \cite[Theorem 1.1]{HH} one can deduce that $\depth R/\overline{I^n}$ is a constant for sufficiently large $n$. Similar to the associated primes, we define {\it the index of depth stability of integral closures of powers} of $I$ by
$$\overline{\dstab}(I) = \min\{n_0 \mid \depth (R/\overline{I^n}) = \depth (R/\overline{I^{n_0}}) \text{ for all } n\geqslant n_0\}.$$

It is hopeless to know $\depth(R/\overline{I^n})$ for each $n$. Thus we are interested in the following problem.
\medskip

\noindent {\bf Problem 2}. Find a sharp bound for $\overline{\dstab}(I)$.

\medskip

In the case $I$ is a monomial ideal of $R$, then the bounds for $\overline{\astab}(I)$ and $\overline{\dstab}(I)$ is obtained in \cite{HT1, Tr1} in terms of the maximal degree of monomial generators of $I$, however they are  not sharp. It is worth mentioning that the similar problem for regularity is studied by Hoa \cite{Hoa}.

In this paper we study two problems above in the case $I = I(G)$ is an edge ideal of a graph $G$. We obtain effective bounds for $\overline{\astab}(I)$ and $\overline{\dstab}(I)$. We also hope that the work will gives an approach to study $\ass(R/\overline{I^n})$ and $\depth(R/\overline{I^n})$ for edge ideals. Before stating the result we recall some definitions. Let $G$ be a graph on the vertex set $V(G)=\{1,\ldots,r\}$ and the edge set $E(G)$.  We associate to $G$ the quadratic square-free monomial ideal
$$I(G) = (x_ix_j \ |\ \{i,j\} \in E(G)) \subseteq R = K[x_1,\ldots, x_r]$$
which is called the edge ideal of $G$.

We use the symbols $\upsilon(G)$ and $\varepsilon_0(G)$ to denote the number of vertices and leaf edges of $G$,  respectively.

Assume that $G_1,\ldots,G_s$ are all connected nonbipartite components of $G$. Let $2k_i-1$ be the minimum length of odd cycles of $G_i$ for every $i = 1,\ldots, s$ and let $2k-1$ be the minimum length of odd cycles of $G$. Define
$$n_0(G)=
\begin{cases}
1  & \text{ if } s = 0,\\
\sum_{i=1}^s(\upsilon(G_i)-\varepsilon_0(G_i)-k_i)  + j + k  & \text{ if } s = 2j \text{ for } j\geqslant 1,\\
\sum_{i=1}^s(\upsilon(G_i)-\varepsilon_0(G_i)-k_i)  + j + 1  & \text{ if } s = 2j+1 \text{ for } j\geqslant 0,
\end{cases}
$$
and
$$\phi_0(G) = \max\{n_0(G') \mid G' \text{ consists of some connected components of } G\}.$$

Martinez-Bernal, Morey and Villarreal \cite{MMV} proved that 
$$\ass (R/\overline{I(G)^n}) = \ass (R/I(G)^n), \text{ for } n\gg 0.$$ 
However, it is still not known a bound of $\overline{\astab}(I(G))$ in terms of the structure of $G$. The first main result of the paper gives a bound of $\overline{\astab}(I(G))$. 

\medskip

\noindent {\bf Theorem $\ref{T3}$}. {\it $\overline{\astab}(I(G)) \leqslant \phi_0(G)$.}

\medskip

Moreover, the equality occurs if $G$ is a pseudoforest (see Theorem \ref{T6}), where a graph is called {\it pseudoforest} if every its connected component has at most one cycle.

\medskip

We next move on to the index of depth stability. Suppose that $G_1,\ldots,G_s$ are all connected bipartite components of $G$ and $G_{s+1},\ldots,G_{s+t}$ are all connected nonbipartite components of $G$. Let $2k_i$ be the maximum length of cycles of $G_i$ ($k_i =1$ if $G_i$ is a tree) for all $i=1,\ldots,s$; and let $2k_i-1$ be the maximum length of odd cycles of $G_i$ for every $i = s+1,\ldots, s+t$; and let $2m-1$ be the minimum length of odd cycles of $G$. Define
$$\phi_1(G)=
\begin{cases}
\upsilon(G)-\varepsilon_0(G)-\sum_{i=1}^{s+t} k_i +1  & \text{ if } t = 0,\\
\upsilon(G)-\varepsilon_0(G)-\sum_{i=1}^{s+t} k_i +j+m  & \text{ if } t = 2j \text{ for } j\geqslant 1,\\
\upsilon(G)-\varepsilon_0(G)-\sum_{i=1}^{s+t} k_i +j+1  & \text{ if } t = 2j+1 \text{ for } j\geqslant 0.
\end{cases}
$$

The second main result of the paper can be stated as follows.

\medskip

\noindent {\bf Theorem $\ref{T4}$}. {\it $\overline{\dstab}(I(G)) \leqslant \phi_1(G)$.}
\medskip

Moreover, the equality occurs if $G$ is a pseudoforest and it contains no cycles of length $4$ (see Theorem \ref{T8}).

\medskip

The first key point to prove our results is the binomial expansion for the integral closures of sums of ideals. Write $G = G' \cup G''$ where $G'$ is the graph consisting of all connected bipartite components of $G$ and $G''$ is the graph consisting of all connected nonbipartite components of $G$. Then, $I(G) = I(G')+I(G'')$. Moreover, $I(G')$ is normally torsion-free by a well-known result of Simis, Vasconcelos and Villarreal \cite{SVV}. The first step to study the integral closures of powers $\overline{I(G)^n}$ is to express it in terms of $I(G')$ and $I(G'')$. Formally, consider two monomial ideals $I$ and $J$ where $I$ is in $A = K[x_1,\ldots,x_s]$ and $J$ is in  $B=K[y_1,\ldots,y_t]$. We want the binomial expansion for $\overline{(I+J)^n}$ in $R = A\otimes_K B = K[x_1,\ldots,x_s,y_1,\ldots,y_t]$ is working. In fact, we prove the following result.

\medskip

\noindent {\bf Theorem $\ref{T0}$}. {\it Let $I$ be a square-free monomial ideal in $A$ and $J$ a monomial ideal in $B$. Assume that $I$ is normally torsion-free. Then,
$$\overline{(I+J)^n} = \sum_{i=0}^n I^i \overline{J^{n-i}}.$$
}

It allows us to use the results about the associated primes and depth of ordinary powers of edge ideals (see \cite{MMV, Tr1}) for the integral closures of powers of edge ideals. The next step is to find the condition under which $\mi \in \ass(R/\overline{I(G)^n})$ where $\mi = (x_1,\ldots,x_r)$ is the maximal homogeneous ideal of $R$. At first glance, Martinez-Bernal, Morey and Villarreal \cite{MMV} proved that every connected component of $G$ must be nonbipartite. The second key point in the proofs of Theorems \ref{T3} and \ref{T4} is the information on how big of such an $n$ (see Lemma $\ref{A3}$). The approach is to describe the integral closures of a monomial ideal in terms of its Newton polyhedron, so that we can use techniques from the theory of convex polyhedra to study our problems.

\medskip

The paper is organized as follows. In the next section, we collect notations and terminology used in the paper, and recall a few auxiliary results. In Section $2$, we prove the binomial expansion for integral closures of powers of monomial ideals (Theorem \ref{T0}).  In Section $3$ we get an upper bound of $\overline{\astab}(I(G))$ (Theorem \ref{T3}). In section $4$ we get an upper bound of $\overline{\dstab}(I(G))$ (Theorem \ref{T4}). In the last section, we prove Theorems \ref{T6} and \ref{T8}, which say that the bounds obtained in Theorems \ref{T3} and \ref{T4} are sharp.

\section{Preliminary} We shall follow standard notations and terminology from usual texts in the research area (cf. \cite{BM,E, MS}). For a nonnegative integer $r$, we denote the set $\{1,\ldots,r\}$ by $[r]$.

\subsection{Simplicial complexes}

A simplicial complex $\Delta$ over the vertex set $V=[r]$ is a collection of subsets of $V$ such that if $F \in \Delta$ and $G\subseteq F$ then $G\in\Delta$. Elements of $\Delta$ are called faces. Maximal faces (with respect to inclusion) are called facets. For $F \in \Delta$, the dimension of $F$ is defined to be $\dim F = |F|-1$. The empty set, $\emptyset$, is the unique face of dimension $-1$, as long as $\Delta$ is not the void complex $\{\}$ consisting of no subsets of $V$. The dimension of $\Delta$ is $\dim \Delta = \max\{\dim F \mid F\in \Delta\}$. 

For a subset $\tau = \{j_1,\ldots,j_i\}$ of $V$, denote $\x^\tau = x_{j_1} \cdots x_{j_i}$.  The Stanley-Reisner ideal of $\Delta$ is defined to be the square-free monomial ideal
$$I_{\Delta} = (\x^\tau \mid \tau \subseteq V \text{ and } \tau \notin \Delta) \ \text{ in } R = K[x_1,\ldots,x_r]$$
and the {\it Stanley-Reisner} ring of $\Delta$ to be the quotient ring $k[\Delta] = R/I_{\Delta}$.  This provides a bridge between combinatorics and commutative algebra (see \cite{S}). 

Note that if $I$ is a square-free monomial ideal, then it is a Stanley-Reisner ideal of the simplicial complex $\Delta(I)= \{\tau \subseteq V \mid \x^\tau\not \in I\}$. When $I$ is a monomial ideal (maybe not square-free) we also use $\Delta(I)$ to denote the simplicial complex corresponding to the square-free monomial ideal $\sqrt{I}$. 

Let $\F(\Delta)$ denote the set of all facets of $\Delta$. Then, $I_\Delta$ has the minimal primary decomposition (see \cite[Theorem 1.7]{MS}):
$$I_\Delta = \bigcap_{F\in \F(\Delta)} (x_i\mid i\notin F),$$
and therefore the $n$-th symbolic power of $I_\Delta$ is
\begin{equation}\label{s-power}
I_\Delta^{(n)} = \bigcap_{F\in \F(\Delta)} (x_i\mid i\notin F)^n.
\end{equation}
The ideal $I_\Delta$ is called normally torsion-free if $I_\Delta^{(n)} = I_\Delta^n$ for all $n\geqslant 1$. Note that since
$$I_\Delta^n \subseteq \overline{I_\Delta^{n}} \subseteq I_\Delta^{(n)},$$
so that $I_\Delta^{n} = \overline{I_\Delta^{n}} = I_\Delta^{(n)}$ if $I_\Delta$ is normally torsion-free.

\subsection{Graphs and Edge Ideals} We recall some standard notation and terminology from graph theory here. Let $G$ be a graph. We always use the symbols $V(G)$ and $E(G)$ to denote the vertex set and the edge set of $G$, respectively.

The ends of an edge of $G$ are said to be incident with the edge, and vice versa. Two vertices which are incident with a common edge are adjacent, and two distinct adjacent vertices are neighbors. The set of neighbors of a vertex $v$ in $G$ is denoted by $N_G(v)$ and  set $N_G[v] = N_G(v) \cup\{v\}$. The degree of a vertex $v$ in $G$, denoted by $\deg_G(v)$, is  the number of neighbours of $v$ in $G$. 

In a graph, a {\it leaf} is a vertex of degree one and a {\it leaf edge} is an edge incident with a leaf. We use the symbols $\upsilon(G)$ and $\varepsilon_0(G)$ to denote the number of vertices and leaf edges of $G$,  respectively.

For a subset $S\subseteq V(G)$ of the vertices in $G$, define $G\setminus S$ to be the subgraph of $G$ with the vertices in $S$ (and their incident edges) deleted. For two graphs $G'$ and $G''$,  the union of them, denoted by $G' \cup G''$, is the graph $G$ with $V(G) = V(G') \cup V(G'')$ and $E(G) = E(G') \cup E(G'')$.

A subgraph $C$ with $k$ vertices of $G$ is called a {\it cycle} of $G$ if it has exactly $k$ vertices and $k$ edges of the form
$$V(C) = \{v_1,\ldots,v_k\}, \text{ and } E(C) = \{\{v_1,v_2\}, \{v_2,v_3\},\ldots, \{v_{k-1},v_k\}, \{v_k,v_1\}\}.$$
A cycle is {\it even} or {\it odd} if its length is even or odd, respectively.

A connected graph is called a {\it tree} if it contains no cycles, and it is called a {\it unicyclic} graph if it contains exactly one cycle. A graph is called {\it pseudoforest} if every its connected component is either a tree or a unicyclic graph.

Assume that $V(G)=\{1,\ldots,r\}$.  We associate to $G$ the quadratic square-free monomial ideal
$$I(G) = (x_ix_j \ |\ \{i,j\} \in E(G)) \subseteq R = K[x_1,\ldots, x_r]$$
which is called the edge ideal of $G$.

An independent set in $G$ is a set of vertices no two of which are adjacent to each other. An independent set $S$ in $G$ is maximal (with respect to set inclusion) if the addition to $S$ of any other vertex in the graph destroys the independence. Let $\Delta(G)$ be the set of independent sets of $G$. Then $\Delta(G)$ is a simplicial complex and this complex is the so-called independence complex of $G$; and facets of $\Delta(G)$ are just maximal independent sets of $G$. Note that $$I(G) = I_{\Delta(G)}.$$

A set $C\subseteq V(G)$ is a vertex cover of $G$ if every edge has an end in $C$. A minimal vertex cover is a vertex cover which is minimal with respect to inclusion. Note that, $C$ is a vertex cover if and only if $V(G)\setminus C$ is an independent set of $G$. Let us denote by $\Gamma(G)$ the set whose elements are the minimal vertex cover sets of $G$. Then, the decomposition $(\ref{s-power})$ can rewrite as
\begin{equation}\label{s-power-cover}
I(G)^{(n)} = \bigcap_{C \in \Gamma(G)} (x_i\mid i\in C)^n.
\end{equation}

The graph $G$ is bipartite if its vertex set can be partitioned into two subsets, say $X$ and $Y$, so that every edge has one end in $X$ and one end in $Y$ ; such a partition $(X, Y )$ is called a bipartition of $G$. It is well-known that $G$ is bipartite if and only if $G$ contains no odd cycle (see \cite[Theorem $4.7$]{BM}). 

We  have a nice algebraic characterization of bipartite graphs as follows.

\begin{lem}\label{L03} {\rm (\cite[Theorem 5.9]{SVV})} Let $G$ be a graph. Then, $I(G)$ is normally torsion-free if and only if $G$ is bipartite.
\end{lem}

\subsection{Integral closures of ideals}

Throughout of this paper we let $\e_1,\ldots,\e_r$ be the canonical basis of $\R^r$. For a monomial ideal $I$ of $R$, let $G(I)$ denote the minimal system of monomial generators of $I$. For a vector $\alb =(\al_1,\ldots,\al_r)\in\N^r$, we denote by $\x^{\alb} = x_1^{\alb_1}\cdots x_r^{\al_r}$ a monomial of $R$. We also denote by $\R_{+}$ the set of nonnegative real numbers.

The {\it integral closure} of an arbitrary ideal $J$ of $R$ is the set of elements $x$ in $R$ that satisfy an integral relation
$$x^n + a_1x^{n-1} + \cdots + a_{n-1}x + a_n = 0$$
where $a_i \in J^i$ for $i = 1,\ldots, n$. It is denoted by $\overline J$ and it is an ideal.

The behavior of the sequence $\{\ass(R/ \overline{J^n})\}_{n\geqslant 1}$ is very nice, namely it is increasing.

\begin{lem}\label{HIO}{\rm (See \cite[Proposition $3.9$]{MA})} Let $J$ be an arbitrary ideal of $R$. Then the sequence $\{\ass(R/ \overline{J^n})\}_{n\in\N}$ is increasing.
\end{lem}

For monomial ideals $I$, then $\overline{I}$ is a monomial ideal as well. We can describe the integral closure of a monomial ideal $I$ geometrically via its Newton polyhedron.

\begin{defn} Let $I$ be a monomial ideal of $R$. We define
\begin{enumerate}
\item  For a subset $A \subseteq R$, the exponent set of $A$ is $E(A) = \{\alb \mid \x^{\alb}\in A\} \subseteq  \Z^r$.
\smallskip
\item The Newton polyhedron of $I$ is $NP(I) = \conv\{E(I)\}$, the convex hull of the exponent set of $I$ in
the space $\R^r$.
\end{enumerate}
\end{defn}

The bridge between the Newton polyhedron of $I$ and its integral closure is given by the well-known equation:
\begin{equation}\label{EN1}
E(\overline I) = NP(I) \cap \Z^r = \{\alb \in \N^r \mid \x^{n\alb} \in I^n \text{ for some } n\geqslant 1\}.
\end{equation}

The Newton polyhedron of the power $I^n$ for $n\geqslant 1$ is related to $NP(I)$ by
\begin{equation}\label{EN2}
NP(I^n) =nNP(I)= n\conv\{E(I)\} +\R_{+}^r.
\end{equation}

We can describe the Newton polyhedron by a system of linear inequalities in $\R^r$ as follows.

\begin{lem}\label{NPH} The Newton polyhedron $NP(I)$ is the set of solutions of a system of inequalities of the form
$$\begin{cases}
\left<\a_j, \x\right> \geqslant 1,\  j=1,\ldots,q,\\
x_1\geqslant 0, \ldots, x_r \geqslant 0,
\end{cases}
$$
such that each hyperplane with the equation $\left<\a_j,\x\right> = 1$ defines a facet of $NP(I)$, which contains $s_j$ affinely independent points of $E(G(I))$ and is parallel to $r - s_j$ vectors of the canonical basis; and $\zv \ne \a_j \in \Q^r \cap \R_+^r$ for all $j = 1, . . . , q$.
\end{lem}
\begin{proof} Follows from \cite[Lemma 1.2]{HT1}.
\end{proof}

\begin{lem} \label{I1} $\overline{I} \cdot \overline{J}\subseteq \overline{IJ}$ for all monomial ideals $I$ and $J$ of $R$.
\end{lem}
\begin{proof} Let $f \in \overline{I} \cdot \overline{J}$ be a monomial, so that $f=uv$ for some monomials $u\in \overline{I}$ and $v\in \overline{J}$. By Equation $(\ref{EN1})$, one has $u^m \in I^m$ and $v^n\in J^n$ for some positive integers $m$ and $n$. It follows that $u^{mn}\in I^{mn}$ and $v^{mn}\in J^{mn}$, and therefore $f^{mn} = u^{mn}v^{mn}\in I^{mn}J^{mn} = (IJ)^{mn}$. By Equation ($\ref{EN1}$) again, we get $f\in\overline{IJ}$. It follows that $\overline I \cdot \overline J \subseteq{\overline{IJ}}$, as required.
\end{proof}

For a subset $F\subseteq [r]$, we set $R_F=R[x_i^{-1}\mid i\in F]$ and $I_F = IR_F \cap R$. Obviously,
$$(\overline{I^n})_F = \overline{(I_F)^n}, \text{ for all } n\geqslant 1.$$

The following result allows us to do induction on the number of variables when studying the associated primes of monomial ideals.

\begin{lem}\label{DI}\cite[Lemma 11]{Tr1} Let $I$ be a monomial ideal of $R$. Then,
$$\ass(R/\overline{I^n}) \setminus \{\mi\} = \bigcup_{i=1}^n \ass(R/\overline{(I_{\{i\}})^n}), \text{ for all } n\geqslant 1.$$
\end{lem}

\section{Binomial expansion of integral closures of powers of ideals}

In this section, we let $A=K[x_1,\ldots,x_s]$, $B = K[y_1,\ldots,y_t]$ and 
$$R = A \otimes_K B = K[x_1,\ldots,x_s,y_1,\ldots,y_t],$$
be polynomial rings over a field $K$, where $\{x_1,\ldots,x_s\}$ and $\{y_1,\ldots,y_t\}$ are two disjoint sets of variables. For simplicity, with two ideals $I$ of $A$ and $J$ of $B$ we still denote $IR$ and $JR$ by $I$ and $J$, respectively.  

Since we have the binomial expansions for ordinary powers and symbolic powers (see \cite[Theorem 3.4]{HHTT}), it is natural to ask whether or not we have the same result for integral closures of powers of monomial ideals, namely:
\begin{equation}\label{BE}
\overline{(I+J)^n} = \sum_{i=0}^n \overline{I^i}\cdot \overline{J^{n-i}}.
\end{equation}

The aim of this section is to prove that we have this expansion if $I$ is normally torsion-free. Namely,

\begin{thm}\label{T0} Let $I$ be a square-free monomial ideal in $A$ and $J$ a monomial ideal in $B$. Assume that $I$ is normally torsion-free. Then,
$$\overline{(I+J)^n} = \sum_{i=0}^n I^i \overline{J^{n-i}}.$$
\end{thm}
\begin{proof} By Lemma \ref{I1} we have 
$$I^i \overline{J^{n-i}} = \overline{I^i}\cdot \overline{J^{n-i}} \subseteq \overline{I^i J^{n-i}} \subseteq \overline{(I+J)^n}, \text{ for all } i = 0,\ldots,n,$$
so that 
$$ \sum_{i=0}^n I^i \overline{J^{n-i}} \subseteq \overline{(I+J)^n}.$$

In order to prove the reverse inclusion, let $f \in \overline{(I+J)^n}$ be a monomial, so that $f = uv$ where $u$ and $v$ are monomials with $u\in A$ and $v\in B$.  Let $\Delta$ be the simplicial complex on $[r]$ corresponding to the square-free monomial ideal $I$. Suppose that $k$ is the largest integer such that $u\in I^k$.  

From Formula (\ref{s-power}) we have 
$$I^j = I^{(j)} =\bigcap_{F\in \F(\Delta)} (x_i\mid i\notin F)^j, \text{ for all } j\geqslant 1.$$ 

It follows that for any $j\geqslant 1$ and $\btb = (\beta_1,\ldots,\beta_s)\in\N^s$, we have
\begin{equation}\label{in-power}
\x^{\btb} \in I^j \ \text{ if and only if } \sum_{i\notin F} \beta \geqslant j, \ \text{ for all } F\in\F(\Delta).
\end{equation}

Assume that $u = \x^{\alb}$ for $\alb = (\alpha_1,\ldots,\alpha_s)\in\N^s$. By the maximality of $k$, Formula (\ref{in-power}) deduces that there is $F \in \F(\Delta)$ such that
\begin{equation}\label{max-hp}
\sum_{i\notin F} \alpha_i = k.
\end{equation}

Since $uv \in \overline{(I+J)^n}$, by Equation $(\ref{EN1})$, one has $f^m \in (I+J)^{mn}$ for some $m\geqslant 1$. Hence, $u^m\in I^p$ and $v^m \in J^{mn-p}$, for some $0\leqslant p \leqslant mn$.

On the other hand, since $u^m\in I^p$, by Equations (\ref{in-power}) and (\ref{max-hp}) we deduce that
$$km = \sum_{i\notin F} m\alpha_i \geqslant p.$$
Thus, $v^m \in J^{mn - p} \subseteq J^{mn - km} = (J^{n-k})^m$. Together with Lemma \ref{EN1}, it yields $v \in \overline{J^{n-k}}$. Therefore, $uv \in I^k \overline{J^{n-k}}$, and therefore
$$\overline{(I+J)^n} \subseteq \sum_{i=0}^n I^i \overline{J^{n-i}},$$
and the theorem follows.
\end{proof}

The assumption that $I$ is normally torsion-free in Theorem \ref{T0} is crucial. In fact, we cannot have the expansion $(\ref{BE})$ even when $I$ and $J$ are edge ideals as the following example.

\begin{exm} Let $I = (ab,bc,ca)$ be an ideal of $A = K[a,b,c]$ and $J = (xy,yz,zx)$ an ideal of $B = K[x,y,z]$. 

Note that $I$ and $J$ are edge ideals of disjoint triangles.  Direct computation with Macaulay2 shows that $\overline{(I+J)^3} \ne \sum_{i=0}^3 \overline{I^i}\cdot \overline{J^{n-i}}$.
\end{exm}

Theorem \ref{T0} has two following interesting consequences, they play a key role in the paper.

\begin{prop} \label{T1} Let $I$ be a square-free monomial ideal of $A$ and $J$ a monomial ideal of $B$. Assume  that $I$ is normally torsion-free. Then, for all $n\geqslant 1$, we have
$$\ass (R/\overline{(I+J)^n}) = \{\p + \q \mid \p\in \ass(R/I) \text{ and } \q\in \ass R/\overline{J^n}\}.$$
In particular, $\overline {\astab}(I + J) = \overline {\astab}(J)$.
\end{prop}
\begin{proof} For each integer $k\geqslant 0$, set $Q_k = \overline{(I+J)^k}$. By Theorem \ref{T0}, we have
\begin{equation}\label{sum-binom}
Q_k = \sum_{j=0}^k I^j \overline{J^{k-j}}.
\end{equation}

We first claim that
\begin{equation}\label{Q-ass}
\ass(Q_k/Q_{k+1}) = \left\{ \p + \q \mid \p \in \ass(A/I)  \text{ and } \q \in \ass(B/\overline{J^{k+1}})   \right\}.
\end{equation}

Indeed, by Equation (\ref{sum-binom}) and \cite[Proposition 3.3]{HHTT}, we get
\begin{equation}\label{Q0}
Q_k/Q_{k+1} \cong \bigoplus_{j=0}^k (I^j/I^{j+1}) \otimes_K (\overline{J^{k-j}}/\overline{J^{k-j+1}}).
\end{equation}

Observe that if $\p \in \ass(I^j/I^{j+1})$ and $\q \in \ass\left(\overline{J^{k-j}}/\overline{J^{k-j+1}}\right)$, then they are monomial ideals. It follows that $\p$ and $\q$ are generated by variables of $A$ and $B$ respectively, so that $\p +\q$ is a prime ideal of $R$. By combining  Equation $(\ref{Q0})$ and  \cite[Theorem 2.5(ii)]{HHTT}, this fact deduces that
 \begin{equation}\label{Q1}
\ass(Q_k/Q_{k+1}) = \bigcup_{j=0}^k \left\{ \p + \q \mid \p \in \ass(I^j/I^{j+1})  \text{ and } \q \in \ass\left(\overline{J^{k-j}}/\overline{J^{k-j+1}}\right)   \right\}.
\end{equation}

Note that $\ass(R/I^{j+1}) = \ass(R/I)$ since $I$ is normally torsion-free. Together with \cite[Lemma 4.4]{MV} , this fact gives
\begin{equation}\label{Q2}
\ass(I^j/I^{j+1})  = \ass(A/I^{j+1}) = \ass(A/I).
\end{equation}

By \cite[Proposition 4]{Tr1} we have
\begin{equation}\label{Q3}
\ass\left(\overline{J^{k-j}}/\overline{J^{k-j+1}}\right) = \ass\left(B/\overline{J^{k-j+1}}\right)
\end{equation}

From Equations $(\ref{Q1})-(\ref{Q3})$ and Lemma \ref{HIO} we obtain
$$\ass(Q_k/Q_{k+1}) = \left\{ \p + \q \mid \p \in \ass(A/I)  \text{ and } \q \in \ass(B/\overline{J^{k+1}})   \right\},$$
and the claim follows.

We now prove the proposition by induction on $n$. If $n=1$, then by Theorem \ref{T0} we have $\overline{I+J} = I + \overline{J}$. Thus,
$$R/(I+J) \cong A/I \otimes_K B/\overline{J},$$
and thus the proposition follows from \cite[Theorem 2.5(ii)]{HHTT}.

Assune that $n > 1$. From the exact sequence
$$0\longrightarrow Q_{n-1}/Q_n \longrightarrow R/Q_n \longrightarrow R/Q_{n-1}\longrightarrow 0,$$
we get
\begin{equation}\label{H0}
\ass(Q_{n-1}/Q_n)\subseteq \ass(R/Q_n) \subseteq \ass(Q_{n-1}/Q_n) \cup \ass(R/Q_{n-1}).
\end{equation}

On the other hand, by the induction hypothesis we have
$$\ass(R/Q_{n-1}) =  \{\p + \q \mid \p\in \ass(R/I) \text{ and } \q\in \ass R/\overline{J^{n-1}}\}.$$

Together with Claim (\ref{Q-ass}) and Lemma \ref{HIO}, this equation implies
$$\ass(R/Q_{n-1}) \subseteq \ass(Q_{n-1}/Q_n)= \{\p + \q \mid p\in \ass(R/I) \text{ and } \q\in \ass R/\overline{J^{n}}\}.$$

By combining this inclusion with Equation (\ref{H0}), we obtain
$$\ass(R/Q_n) = \ass(Q_{n-1}/Q_n)= \{\p + \q \mid \p\in \ass(R/I) \text{ and } \q\in \ass R/\overline{J^{n}}\},$$
and the proof is complete.
\end{proof}

\begin{prop}\label{T2}  Let $I$ be a square-free monomial ideal of $A$ and $J$ a monomial ideal of $B$. Assume that that $I$ is normally torsion-free. Then, for all $n\geqslant 1$, we have
\begin{align*}
\depth &R/\overline{(I+J)^n} \\
&= \min\limits_{i\in[n-1], \ j\in[n]}\{\depth A/I^{n-i} + \depth B/\overline{J^i}+1, \depth A/I^{n-j+1} + \depth B/\overline{J^j}\}.
\end{align*}
\end{prop}
\begin{proof} Folows from \cite[Theorem 5.3 and 5.10]{HHTT} and Theorem \ref{T0}.
\end{proof}

\section{The index of stability}

In this section, we always assume that $R=K[x_1,\ldots,x_r]$ is a polynomial ring over a field $K$ and $\mi =(x_1,\ldots,x_r)$ is the maximal homogeneous ideal of $R$. For a graph $G$, we also assume that $V(G) = [r]$.

\begin{lem}\label{A1} Let $G$ be a connected nonbipartite graph. Let $2m-1$ be the maximum length of odd cycles of $G$. Then for any $n\geqslant \upsilon(G) -\varepsilon_0(G) - m+1$, there is a monomial $f$ of degree $2n-1$ such that
$\mi = I(G)^n: f$ and $f^2 \in I(G)^{2n-1}$.
\end{lem}
\begin{proof} Let $C$ be an old cycle of $G$ with length $2m-1$. If $C'$ is another cycle of $G$, then $C'$ has an edge $e$ that does not lie on the cycle $C$. Delete this edge from $G$, thereby obtaining a connected subgraph $G'$ of $G$ with $V(G') = V(G)$ and $C$ is still a cycle of $G'$. This process continues until we obtain a connected subgraph $H$ of $G$ such that $V(G)=V(H)$ and $H$ has only one cycle $C$. Let $s =\upsilon(H)-\varepsilon_0(H)-m+1$. 

Since $H$ is a unicyclic nonbiparite graph, by \cite[Lemma $2.3$]{Tr} there is a monomial $g\in R$ such that $\deg g = 2s-1$ and
$x_ig\in I(H)^s$ for all $i=1,\ldots,r$. Because $I(H)\subseteq I(G)$, therefore
\begin{equation}\label{A11}
x_ig\in I(G)^s \text{ for all } i=1,\ldots,r.
\end{equation}
As $G$ is generated by quadric monomials and $\deg g =2s-1$, so $g\notin I(G)^s$. Thus, $\mi= I(G)^s:g$. Since $\upsilon(G)=\upsilon(H)$ and $\varepsilon_0(G) \leqslant \varepsilon_0(H)$, hence $s \leqslant \upsilon(G) -\varepsilon_0(G)-m+1$.

Write $n = s+t$ for some $t\geqslant 0$. Let us choose an quadratic monomial $u$ of $I(G)$ and let $f = u^tg$. We will prove that $f$ is a desired monomial.

Since $\deg f = 2t+\deg g = 2t+2s-1 = 2n-1$, so $f\notin I(G)^n$. Together with the fact $(\ref{A11})$, we conclude that $\mi = I(G)^n:f$.

It remains to prove that $f^2\in I(G)^{2n-1}$. Since $\deg(g)=2s-1$ and $gx_1 \in I(G)^s$, it follows that $gx_1 = f_1\cdots f_s$ where $f_1,\ldots,f_s$ are quadratic monomials of $I(G)$. Without loss of generality, we may assume that $x_1 \mid f_s$ and $f_s = x_1x_j$ for some $1\leqslant j\leqslant r$. It follows that
$$g = f_1\cdots f_{s-1} x_j.$$
Hence, $g^2 = (f_1\cdots f_{s-1})(gx_j) \in I(G)^{s-1}I(G)^s = I(G)^{2s-1}$. Therefore,
$$f^2 = u^{2t} g^2 \in I(G)^{2t}I(G)^{2s-1} = I(G)^{2t+2s-1} = I(G)^{2n-1},$$
and the proof is complete.
\end{proof}

\begin{lem} \label{A2} Let $G$ be a graph. Let $f_1,\ldots,f_{2s}$ be monomials of $R$ with $s\geqslant 1$. Assume that for all $i=1,\ldots,2s$, we have $\deg f_i = 2n_i-1$ and $f_i^2 \in I(G)^{2n_i-1}$ where $n_i\geqslant 1$. Then,
$$f_1\cdots f_{2s} \in \overline{I(G)^n}$$
where $n=n_1+\cdots+n_{2s}-s$.
\end{lem}
\begin{proof} For each $i=1,\ldots,2s$, since $\deg f_i^2 = 2(n_i-1)$ and $f_i^2\in I(G_i)^{2n_i-1}$, together with the fact that $I(G)$ is generated by quadratic monomials we imply that there are $(2n_i-1)$ quadratic monomials, say $\x^{\alb_{i,1}},\ldots,\x^{\alb_{i,2n_i-1}}$, of $I(G)$ such that
$$f_i^2 = \x^{\alb_{i,1}}\cdots \x^{\alb_{i,2n_i-1}}.$$

Let $f =f_1\cdots f_{2s}$. Write $f =\x^{\alb}$ with $\alb\in\N^r$. Then,
$$2\alb = \sum_{i=1}^{2s}\sum_{j=1}^{2i-1}\alb_{i,j},$$
hence
$$\alb = \sum_{i=1}^{2s}\sum_{j=1}^{2i-1}\frac{1}{2}\alb_{i,j}.$$
Since
$$\sum_{i=1}^{2s}\sum_{j=1}^{2n_i-1}\frac{1}{2} = \sum_{i=1}^{2s} \frac{2n_i-1}{2}=\sum_{i=1}^{2s}n_i -s = n,$$
by Formula ($\ref{EN2}$) we conclude that $\x^\alb \in \overline{I(G)^n}$, as required.
\end{proof}

\begin{defn}Let $G$ be a graph with connected components $G_1,\ldots,G_p$ such that all $G_1,\ldots,G_p$ are nonbipartite.  For each $i=1,\ldots,p$, let $2m_i-1$ be the maximum length of odd cycles of $G_i$. Let $2m-1$ be the minimum length of odd cycles of $G$. Let
$$n_1(G)=
\begin{cases}
\upsilon(G)-\varepsilon_0(G)-\sum_{i=1}^p m_i +s+1  & \text{ if } p = 2s+1 \text{ for  } s\geqslant 0,\\
\upsilon(G)-\varepsilon_0(G)-\sum_{i=1}^p m_i +s+m  & \text{ if } p = 2s \text{ for  } s\geqslant 1.
\end{cases}
$$
\end{defn}

Note that $\mi \in\ass(R/\overline{I(G)^n})$ for $n\geqslant 1$ if and only if every connected component of $G$ is nonbipartite (see \cite[Proposition 3.3]{MMV}).  The following lemma gives the information on how big of $n$ such that $\mi \in\ass(R/\overline{I(G)^n})$.

\begin{lem} \label{A3} If every connected component of a graph $G$ is nonbipartite, then we have $\mi\in\ass (R/\overline{I(G)^n})$  for all $n\geqslant n_1(G)$.
\end{lem}

\begin{proof} Let $G_1,\ldots,G_p$ be connected components of $G$. Then, we can rewrite $n_1(G)$ as
$$n_1(G)=
\begin{cases}
\sum_{i=1}^p(\upsilon(G_i)-\varepsilon_0(G_i)-m_i) +s+1  & \text{ if } p = 2s+1 \text{ for  } s\geqslant 0,\\
\sum_{i=1}^p(\upsilon(G_i)-\varepsilon_0(G_i)-m_i) +s+m  & \text{ if } p = 2s \text{ for  } s\geqslant 1.
\end{cases}
$$
We consider two possible cases.
\medskip

{\it Case $1$}: $p=2s+1$. If $s=0$, then lemma follows from Lemma $\ref{A1}$, and then we assume that $s\geqslant 1$. Let $n_i = \upsilon(G_i)-\varepsilon_0(G_i) -m_i+1$ for $i=1,\ldots,2s$ and $n_{2s+1}=n-(n_1+\cdots+n_{2s}-s)$. Then, $n_{2s+1}\geqslant \upsilon(G_{2s+1})-\varepsilon_0(G_{2s+1}) -m_{2s+1}+1$.  

By Lemma $\ref{A1}$, for each $i=1,\ldots,2s+1$, there is a monomial $f_i$ of degree $2n_i-1$ such that
\begin{equation} \label{G1} \mi_i = I(G_i)^{n_i}:f_i \ \text{ and } f_i^2 \in I(G_i)^{2n_i-1}
\end{equation}
where $\mi_i = (x_j \mid j \in V(G_i))$.

Let $f  = f_1\ldots f_{2s+1}$, so that $\deg f = 2(n_1+\cdots+n_{2s+1})-(2s+1) = 2n-1$. It follows $f\notin \overline{I(G)^n}$. We now prove that $\mi = \overline{I(G)^n}:f$, i.e., $f x_i \in \overline{I(G)^n}$ for each $i=1,\ldots,r$. In order to prove $fx_i \in\overline{I(G)^n}$, we may assume that $i\in V(G_1)$. By Formula ($\ref{G1}$) we have $f_1x_i \in I(G)^{n_1}$. Let $J =I(G_2)+\cdots+I(G_{2s+1})$ and
$m = n_2+\cdots+n_{2s+1} - s$. By Lemma $\ref{A2}$ one has $f_2\cdots f_{2s+1} \in \overline {J^m}$. Notice that $n_1+m = n$ and $J\subseteq I(G)$. Together with Lemma $\ref{I1}$ we get
$$fx_i = (f_1x_i)f_2\cdots f_{2s+1} \in \overline{I(G_1)^{n_1}}\cdot  \overline{J^m}\subseteq \overline{I(G_1)^{n_1}J^m} \subseteq \overline{I(G)^{n_1+m}} =  \overline{I(G)^n}.$$
It follows $\mi = \overline{I(G)^n}:f$ and the lemma holds for this case.

{\it Case $2$}: $p =2s$. The proof is almost the same as the previous case. We may assume that $G_1$ has a cycle, say $C$, of length $2m-1$. Let $n_i = \upsilon(G_i)-\varepsilon_0(G_i) -m_i+1$ for $i=1,\ldots,2s-1$ and $n_{2s}=n-(m+n_1+\cdots+n_{2s-1}-s)$. Then we have $n_{2s}\geqslant \upsilon(G_{2s})-\varepsilon_0(G_{2s}) -m_{2s}+1$.  By Lemma $\ref{A1}$, for each $i=1,\ldots,2s$, there is a monomial $f_i$ of degree $2n_i-1$ such that
\begin{equation} \label{G2} \mi_i = I(G_i)^{n_i}:f_i \ \text{ and } f_i^2 \in I(G_i)^{2n_i-1}
\end{equation}
where $\mi_i = (x_j \mid j \in V(G_i))$.

Suppose that $V(G_1) =\{x_1,\ldots,x_q\}$ for some $1\leqslant q < r$ and the cycle $C$ is $x_1,\ldots,x_{2m-1}$. Let $g = x_1x_2\cdots x_{2m-1}$. Then $\deg g = 2m-1$ and $g^2 \in I(C)^{2m-1}$. Let $f = gf_1\ldots f_{2s}$, so that $\deg f = (2m-1)+2(n_1+\cdots+n_{2s})-2s = 2n-1$. It follows $f\notin \overline{I(G)^n}$. We now prove that $\mi = \overline{I(G)^n}:f$, i.e., $f x_i \in \overline{I(G)^n}$ for every $i=1,\ldots,r$. 

For $i\in V(G_1)$,  by Formula ($\ref{G2}$) we have $f_1x_i \in I(G)^{n_1}$. Let $J =I(C)+I(G_2)+\cdots+I(G_{2s})$ and
$t = m+n_2+\cdots+n_{2s} - s$. By Lemma $\ref{A2}$ one has $gf_2\cdots f_{2s} \in \overline {J^t}$. Notice that $n_1+t = n$ and $J\subseteq I(G)$. Together with Lemma $\ref{I1}$ we get
$$fx_i = (f_1x_i)gf_2\cdots f_{2s} \in \overline{I(G_1)^{n_1}}\cdot  \overline{J^t}\subseteq \overline{I(G_1)^{n_1}\cdot  J^t} \subseteq \overline{I(G)^{n_1+t}} =  \overline{I(G)^n}.$$

For  $i\in V(G_j)$ for some $2\leqslant j \leqslant 2s$, we prove $fx_i \in \overline{I(G)^n}$ by the same way. Thus, $\mi = \overline{I(G)^n}:f$, and the proof  is complete.
\end{proof}

\begin{lem}\label{UPC} Let $G$ be a connected nonbipartite graph and $v$ a vertex of $G$. Let $G_1,\ldots, G_s$ be the connected nonbipartite components of $G\setminus N_G[v]$. Let $2k_i-1$ be the minimum length of odd cycles of $G_i$ for every $i = 1,\ldots,s$; and let $2k-1$ be the minimum length of odd cycles of $G$. Then,
\begin{enumerate}
\item $\upsilon(G_i) - \varepsilon_0(G_i) -k_i \leqslant \upsilon(G)-\varepsilon_0(G)-k-1$, for $i=1,\ldots,s$.
\item If $s\geqslant 2$, we have
$$\sum_{i=1}^s (\upsilon(G_i) - \varepsilon_0(G_i) -k_i) \leqslant (\upsilon(G)-\varepsilon_0(G)-k) - k_1-(2s-3).$$
\end{enumerate}
\end{lem}
\begin{proof} For every $i=1,\ldots,s$, $G_i$ is not an edge, so $\sum_{i=1}^s (\upsilon(G_i) - \varepsilon_0(G_i)$ is the number of  non-leaf vertices of $\bigcup_{i=1}^s G_i$.  Note also that non-leaf vertices of $\bigcup_{i=1}^s G_i$ are those of  $G$. Since the set $N_G[v]$  contains at least one non-leaf vertex of $G$, it follows that
\begin{equation}\label{UPC1}
\sum_{i=1}^s (\upsilon(G_i) - \varepsilon_0(G_i)) + 1 \leqslant \upsilon(G) - \varepsilon_0(G).
\end{equation}

In particular, for each $i=1,\ldots,s$, one has $\upsilon(G_1) - \varepsilon_0(G_1) + 1 \leqslant \upsilon(G) - \varepsilon_0(G)$. Note that $k \leqslant k_i$, so
$$(\upsilon(G_i) - \varepsilon_0(G_i)-k_i) + 1 \leqslant \upsilon(G) - \varepsilon_0(G)-k,$$
and $(1)$ follows.

Assume that $s \geqslant 2$. Since $k_i\geqslant k\geqslant 2$ for all $i=1,\ldots,s$, we deduce that
$$\sum_{i=3}^s k_i \geqslant 2(s-2).$$

Together with Inequality (\ref{UPC1}), it gives
\begin{align*}
\sum_{i=1}^s &(\upsilon(G_i) - \varepsilon_0(G_i) -k_i) +k_1+2s-3 = \sum_{i=1}^s (\upsilon(G_i) - \varepsilon_0(G_i)) - k_2 -\sum_{i=3}^s k_i +2s-3 \\
&\leqslant \upsilon(G)-\varepsilon_0(G)-1 -k - 2(s-2)+2s-3 \leqslant \upsilon(G)-\varepsilon_0(G) - k,
\end{align*}
and $(2)$ follows. The proof is complete.
\end{proof}

We will establish an upper bound for $\overline{\astab}(I(G))$ by starting with the definition of  the invariant $\phi_0(G)$.

\begin{defn} \label{DEF} Let $G$ be a graph. Assume that $G_1,\ldots,G_s$ be all connected nonbipartite components of $G$. Let $2k_i-1$ be the minimum length of odd cycles of $G_i$ for every $i = 1,\ldots, s$ and let $2k-1$ be the minimum length of odd cycles of $G$. Then, we define
$$n_0(G)=
\begin{cases}
1  & \text{ if } s = 0,\\
\sum_{i=1}^s(\upsilon(G_i)-\varepsilon_0(G_i)-k_i)  + j + k  & \text{ if } s = 2j \text{ for } j\geqslant 1,\\
\sum_{i=1}^s(\upsilon(G_i)-\varepsilon_0(G_i)-k_i)  + j + 1  & \text{ if } s = 2j+1 \text{ for } j\geqslant 0,
\end{cases}
$$
and
$$\phi_0(G) = \max\{n_0(G') \mid G' \text{ consists of some connected components of } G\}.$$
\end{defn}

\begin{lem}\label{sub-ineq} Let $G$ be a graph and $v$ a vertex of $G$. Then, $\phi_0(G\setminus N_G[v]) \leqslant \phi_0(G)$.
\end{lem}
\begin{proof} Suppose that $G_1,G_2,\ldots,G_s$ are all connected nonbipartite components of $G$. We may assume that $v\in V(G_1)$, so that $N_G[v] = N_{G_1}[v]$. If $G_1\setminus N_G[v]$ is bipartite, then all connected nonbipartite components of $G\setminus N_G[v]$ are just $G_2,\ldots, G_s$. By Definition \ref{DEF} we have $\phi_0(G_1\setminus N_G[v]) \leqslant \phi_0(G)$, and the lemma holds for this case.

Assume that $G_1\setminus N_G[v]$ is not bipartite and let $H_1,\ldots,H_t$ be all connected nonbipartite components of $G_1\setminus N_G[v]$. Then, all connected nonbipartite components of $G\setminus N_G[v]$ are just
$$H_1,\ldots, H_t, G_2,\ldots,G_s.$$

Without loss of generality, we may assume that $$\phi_0(G\setminus N_G[v]) = n_0(H_1 \cup \cdots H_p \cup G_2\cup \ldots\cup G_q),$$
for some $0\leqslant p \leqslant t$ and $1\leqslant q \leqslant s$.

Then, the inequality $\phi_0(G\setminus N_G[v]) \leqslant \phi_0(G)$ is obvious if $p = 0$ by Definition \ref{sub-ineq}, so that we assume that $p \geqslant 1$.

For simplicity, let
$$G' = H_1 \cup \cdots H_p \cup G_2\cup \ldots\cup G_q, \ \text{ and } G'' = G_1 \cup G_2\cup \ldots\cup G_q.$$

Now in order to prove $\phi_0(G\setminus N_G[v])\leqslant \phi_0(G)$ it suffices to show that 
\begin{equation}\label{main-ineq}
n_0(G') \leqslant n_0(G'').
\end{equation}

Let $2k_i-1$ be the minimum length of odd cycles of $G_i$ for $i=1,\ldots,s$; let $2l_i-1$ be the minimum length of odd cycles of $H_i$ for $i=1,\ldots, q$; let $2k-1$ be the minimum length of odd cycles of $G''$; and let $2l-1$ be the minimum length of odd cycles of $G'$. Observe that $k\leqslant l$.

\medskip

In order to prove $(\ref{main-ineq})$, we consider two possible cases.

\medskip

{\it Case $1$}: $q = 2j$ for $j\geqslant 1$. In this case,

\begin{equation}\label{MIn4}
n_0(G'') = \sum_{i=1}^q(\upsilon(G_i)-\varepsilon_0(G_i)-k_i)  + j + k.
\end{equation}

We next consider two subcases.
\medskip

{\it Subcase $1$}: $p = 2e+1$ for $e\geqslant 0$. Then, $p + q -1 = 2(j+e)$. Hence,

\begin{equation}\label{MIn5}
n_0(G')  =\sum_{i=1}^p \upsilon(H_i)-\varepsilon_0(H_i)-l_i) +\sum_{i=2}^{q} \upsilon(G_i)-\varepsilon_0(G_i)-k_i) + j + e +l.
\end{equation}
From Formulas (\ref{MIn4}) and (\ref{MIn5}), The inequality $n_0(G')\leqslant n_0(G'')$ is equivalent to
\begin{equation}\label{MIn5e}
\sum_{i=1}^p \upsilon(H_i)-\varepsilon_0(H_i)-l_i) + e +l \leqslant \upsilon(G_1)-\varepsilon_0(G_1) -k_1 +k.
\end{equation}

If $k = k_1$, since $l \leqslant l_1$, we have the inequality above holds true whenever
$$\sum_{i=1}^p (\upsilon(H_i)-\varepsilon_0(H_i)-l_i) + e +l_1 \leqslant \upsilon(G_1)-\varepsilon_0(G_1).$$
As this inequality follows from Lemma \ref{UPC}, so the inequality (\ref{main-ineq}) holds true.

Assume that $k \ne k_1$, so that $k = k_i$ for some $i\geqslant 2$. It follows that  $l = k_i = k$. Hence,  the inequality (\ref{MIn5e}) holds true if
$$\sum_{i=1}^p (\upsilon(H_i)-\varepsilon_0(H_i)-l_i) + e\leqslant \upsilon(G_1)-\varepsilon_0(G_1) -k_1.$$
It follows from Lemma \ref{UPC}, so the inequality (\ref{main-ineq}) holds true. Thus, the desired inequality is proved for this subcase.

{\it Subcase $2$}: $p = 2(e+1)$ for $e\geqslant 0$. Then, $p + q -1 = 2(j+e)+1$. Hence,

\begin{equation}\label{MIn6}
n_0(G')  =\sum_{i=1}^p \upsilon(H_i)-\varepsilon_0(H_i)-l_i) +\sum_{i=1}^{q-1} \upsilon(G_i)-\varepsilon_0(G_i)-k_i) + j + e +1.
\end{equation}

From Formulas (\ref{MIn4}) and (\ref{MIn6}), the inequality $n_0(G')\leqslant n_0(G'')$ is equivalent to the following one
\begin{equation}\label{DE002}
\sum_{i=1}^p (\upsilon(H_i)-\varepsilon_0(H_i)-l_i) + e +1 \leqslant \upsilon(G_1)-\varepsilon_0(G_1) -k_1 + k.
\end{equation}

Since $p\geqslant 2$ and $k\geqslant 2$,  the inequality $(\ref{DE002})$ follows from Lemma \ref{UPC}, and therefore the inequality (\ref{main-ineq}) is proved for this subcase.

\medskip

{\it Case $2$}: $q = 2j+1$ for $j \geqslant 0$. The proof is the same as in the case $1$ above so that we skip it, and thus the lemma follows.
\end{proof}

We are now in position to prove the main result of this section.

\begin{thm}\label{T3} Let $G$ be a graph. Then, $\overline{\astab}(I(G)) \leqslant \phi_0(G)$.
\end{thm}

\begin{proof} We prove the theorem by induction on $\upsilon(G)$. If $\upsilon(G) \leqslant 2$, then $G$ is bipartite. Hence, $I(G)$ is normally torsion-free by Lemma $\ref{L03}$, and hence the theorem holds true in this case.

Assume that $\upsilon(G) \geqslant 3$.  Let $G'$ be the subgraph of $G$ consisting of all connected bipartite components of $G$ and let $G''$ be the subgraph of $G$ consisting of all connected nonbipartite components of $G$. Then, $I(G) = I(G')+I(G'')$. On the other hand, $I(G')$ is normally torsion-free by Lemma $\ref{L03}$, so $\overline{\astab}(I(G)) = \overline{\astab}(I(G''))$ by Proposition \ref{T1}. On the other hand, by Definition \ref{DEF} we have $\phi_0(G) = \phi_0(G'')$, so that we may assume that $G = G''$, that is the case all connected components of $G$ are nonbipartite.

By Lemma $\ref{A3}$ we deduce that $\mi \in\ass R/\overline{I(G)^n}$ for all $n\geqslant \phi_0(G)$. Together with Lemma \ref{DI}, we obtain
$$\ass R/\overline{I(G)^n} =\{\mi\}\cup \bigcup_{i=1}^r \ass R/\overline{(I(G)_{\{i\}}) ^n}\ \text{ for all } n\geqslant \overline n_0(G).$$
In particular, 
\begin{equation}\label{as01}
\overline{\astab}(I(G)) \leqslant \max\{\phi_0(G),\overline{\astab} (I(G)_{\{1\}}),\ldots, \overline{\astab} (I(G)_{\{r\}})\}.
\end{equation}

We next claim that  for every $i=1,\ldots,r$, we have $\overline{\astab} (I(G)_{\{i\}}) \leqslant \phi_0(G)$. Indeed, fix an integer $1\leqslant i \leqslant r$, we may assume that $N_G(i)=\{1,\ldots,p\}$. Then,  $I(G)_{\{i\}} = (x_1,\ldots,x_p)+ I(H)$ where $H = G\setminus N_G[i]$. Now by Proposition $\ref{T1}$ we have $\overline{\astab}(I(G)_{\{i\}}) =\overline {\astab}(I(H))$. 

On the other hand, since $\upsilon(H) < \upsilon(G)$, by the induction hypothesis we have $\overline {\astab}(H) \leqslant \phi_0 (H)$. By Lemma \ref{sub-ineq}, we have $\phi_0(H) \leqslant \phi_0(G)$. Hence, $\overline {\astab}(I(H)) \leqslant \phi_0(G)$, and the claim follows.

Now, together the claim with Inequality (\ref{as01}) we get $\overline{\astab}(I(G)) \leqslant \phi_0(G)$, and the proof is complete.
\end{proof}

\begin{cor} Let $G$ be a connected nonbipartite graph. Assume the minimum length of odd cycles of $G$ is $2k-1$. Then, $\overline{\astab}(I(G)) \leqslant \upsilon(G) - \varepsilon_0(G) - k + 1$.
\end{cor}
\begin{proof} Since $\phi_0(G) = \upsilon(G) - \varepsilon_0(G) - k + 1$ in this case, the corollary follows from Theorem \ref{T3}.
\end{proof}

\section{the index of depth stability}

In this section we will establish a bound for $\overline{\dstab}(I(G))$ for all simple graphs $G$. As usual, we let $R = K[x_1,\ldots,x_r]$ and $\mi = (x_1,\ldots,x_r)$ the maximal homogeneous ideal of $R$. For any graph with symbol $G$, we assume that $V(G) = [r]$.

\medskip

First we have the following lemma.

\begin{lem}\label{D1} Let $G$ be a graph. Assume that $G = G' \cup G''$ is a disjoint union of two induced subgraphs $G'$ and $G''$ where every connected component of $G'$ is bipartite and every connected component of $G''$ is nonbipartite. Let $s$ be the number of connected bipartite components of $G$. Then,
\begin{enumerate}
\item $\depth R/\overline{I(G)^n} \geqslant s$ for all $n\geqslant 1$.
\item $\depth R/\overline{I(G)^n} = s$ if and only if $n\geqslant \dstab(I(G')) +\overline{\dstab}(I(G''))-1$.
\item $\overline{\dstab}(I(G)) = \dstab(I(G')) +\overline{\dstab}(I(G''))-1$.
\end{enumerate}
\end{lem}
\begin{proof} Let $A = K[x_i\mid i\in V(G')]$ and $B = K[x_i\mid i \in V(G'')]$. For simplicity, set $I = I(G')$ and $J = I(G'')$. Then, $I$ and $J$ are monomial ideals in $A$ and $B$ respectively. Note also that $I(G) = I+J$ and $s$ is the number of connected components of $G'$.

By \cite[Theorem 4.4]{Tr} we have
\begin{equation}\label{GG1}
\depth A/I^n \geqslant s, \ \text{ for all } n\geqslant 1,
\end{equation}
and
\begin{equation}\label{GG2}
\dstab(I) =\min\{n\geqslant 1\mid \depth R/I(G)^n = s\}.
\end{equation}

Let $\mi_B$ be the maximal homogeneous ideal of $B$. By Lemma \ref{A3} we have $\mi_B \in \ass(B/\overline{J^n})$ for all $n$ large enough. Together with Lemma \ref{HIO} we deduce that
\begin{equation}\label{GG3}
\overline{\dstab}(J) = \min\{n\geqslant 0\mid \mi_B \in \ass(B/\overline{J^n})\} =  \min\{n\geqslant 0\mid  \depth(B/\overline{J^n}) = 0\}.
\end{equation}

On the other hand, since $I$ is normally torsion-free according to Lemma \ref{L03}, by Proposition \ref{T2}, for all $n\geqslant 1$, we have
\begin{align}\label{GG4}
\depth R/\overline{I(G)^n} = \min\limits_{i\in[n-1], \ j\in[n]}\{&\depth A/I^{n-i} + \depth B/\overline{J^i}+1,\\
 &\depth A/I^{n-j+1} + \depth B/\overline{J^j}\}.\notag
\end{align}

For $i\in [n-1]$ and $j\in [n]$, from Formulas (\ref{GG1})-(\ref{GG3}) we deduce that
\begin{equation}\label{GG5}
\depth A/I^{n-i} + \depth B/\overline{J^i}+1 \geqslant s + 1 \text{ and } \depth A/I^{n-j+1} + \depth B/\overline{J^j} \geqslant s.
\end{equation}

From Formula $(\ref{GG4})$ and Inequalities $(\ref{GG5})$, we have $\depth R/\overline{I(G)^n}\geqslant s$ for all $n\geqslant 1$, and thus $(1)$ follows.

Next from $(\ref{GG4})$ and Inequalities $(\ref{GG5})$, we deduce that $\depth R/\overline{I(G)^n} = s$ if and only if 
$$\depth A/I^{n-j+1} = 0 \text{ and }  \depth B/\overline{J^j} = 0 \text{ for some } j \in [n].$$
Together with Formulas (\ref{GG1})-(\ref{GG3}), it is equivalent to
$$n-j+1\geqslant \dstab(I) \text{ and } j \geqslant \overline{\dstab}(J) \text{ for some } j \in [n].$$
It holds true if and only if $(n-j+1)+j \geqslant \dstab(I)+ \overline{\dstab}(J)$, or $n\geqslant \dstab(I)+ \overline{\dstab(J)}-1$, and $(2)$ follows.

Finally, since $(3)$ follows from $(2)$, and  the proof  is complete.
\end{proof}

\begin{defn} \label{DEF1} Let $G$ be a graph. Let $G_1,\ldots,G_s$ be all connected bipartite components of $G$ and let $G_{s+1},\ldots,G_{s+t}$ be all connected nonbipartite components of $G$. Let $2k_i$ be the maximum length of cycles of $G_i$ ($k_i =1$ if $G_i$ is a tree) for all $i=1,\ldots,s$; and let $2k_i-1$ be the maximum length of odd cycles of $G_i$ for every $i = s+1,\ldots, s+t$; and let $2m-1$ be the minimum length of odd cycles of $G$. Define
$$\phi_1(G)=
\begin{cases}
\upsilon(G)-\varepsilon_0(G)-\sum_{i=1}^{s+t} k_i +1  & \text{ if } t = 0,\\
\upsilon(G)-\varepsilon_0(G)-\sum_{i=1}^{s+t} k_i +j+m  & \text{ if } t = 2j \text{ for } j\geqslant 1,\\
\upsilon(G)-\varepsilon_0(G)-\sum_{i=1}^{s+t} k_i +j+1  & \text{ if } t = 2j+1 \text{ for } j\geqslant 0.
\end{cases}
$$
\end{defn}

We now ready to prove the main result of this section.

\begin{thm}\label{T4} Let $G$ be a graph. Then, $\overline{\dstab}(I(G)) \leqslant \phi_1(G)$.
\end{thm}

\begin{proof} If $t=0$, i.e. $G$ is a bipartite graph, then $I(G)$ is normally torsion-free by Lemma $\ref{L03}$. Hence, $\overline{\dstab}(I(G)) = \dstab(I(G))$. The theorem follows from \cite[Theorem $4.6$]{Tr}.

If $s = 0$, i.e. every connected component of $G$ is nonbipartite. By Lemma $\ref{A3}$ we deduce that $\depth R/\overline{I(G)^n} =0$, for all $n\geqslant \phi_1 (G)$. Therefore, $\overline{\dstab}(I(G)) \leqslant \phi_1(G)$, and the theorem follows.

Assume that $s\geqslant 1$ and $t\geqslant 1$. In order to prove the theorem it suffices to show $\depth R/\overline{I(G)^n}=s$ for all $n\geqslant \phi_1(G)$.

Let $H = G_1 \cup \cdots \cup G_s$ and $W =G_{s+1}\cup\cdots\cup G_{s+t}$. Then $G = H\cup W$ and $I(G) = I(H)+I(W)$. Let $A = K[x_i\mid i\in V(H)]$ and $B = K[x_i\mid i\in V(W)]$. 
By Lemma \ref{D1} we have
\begin{equation}\label{DS1}
\overline{\dstab}(I(G)) = \dstab(I(H)) + \overline{\dstab}(I(W))-1.
\end{equation}

On the other hand, by \cite[Theorem 4.6]{Tr} we have
\begin{equation}\label{DS2}
\dstab(I(H)) \leqslant \upsilon(H)-\varepsilon_0(H)-\sum_{i=1}^{s} k_i +1 = \phi_1(H).
\end{equation}

Let $\mi_B = (x_i \mid i\in V(W))$ is the maximal homogeneous ideal of $B$. By Lemmas \ref{A3} and $\ref{D1}$, we imply that
\begin{equation}\label{DS3}
\overline{\dstab}(I(W))  = \min\{n\geqslant 1\mid \mi_B \in \ass(B/\overline{I(W)^n}\}\leqslant  \phi_1(W).
\end{equation}

Observe that $\phi_1(G) = \phi_1(H) + \phi_1(W)-1$. Together with Formulas $(\ref{DS1})$ - $(\ref{DS3})$, this equality forces
$$\overline{\dstab}(I(G)) = \dstab(I(H)) + \overline{\dstab}(I(W))-1\leqslant \phi_1(H) + \phi_1(W)-1 = \phi_1(G),$$
and theorem follows.
\end{proof}

\section{Pseudoforests}

In this section we show that the bounds in Theorems \ref{T3} and \ref{T4} are sharp for pseudoforests. Recall that a pseudoforest is a graph in which every connected component is either a tree or a unicyclic graph.

\begin{lem} \label{vertex-cover} Let $G$ be a graph and $C$ a vertex cover of a graph $G$. If $\x^{\alb} \in \overline{I(G)^n}$ with $n\geqslant 1$, then
$$\sum_{i\in C} \alpha_i \geqslant n.$$
\end{lem}
\begin{proof} We may assume that $C$ is a minimal vertex cover. Since $\overline{I(G)^n} \subseteq I(G)^{(n)}$, one has $\x^{\alb} \in I(G)^{(n)}$. On the other hand, by Formula $(\ref{s-power-cover})$, $I(G)^{(n)}$ has the primary decomposition as folows
$$I(G)^{(n)} = \bigcap_{C \in \Gamma(G)} (x_i\mid i\in C)^n,$$
where $\Gamma(G)$ is the set of all minimal vertex cover sets of $G$.  In particular, it yields $\x^{\alb} \in  (x_i\mid i\in C)^n$, and so
$$\sum_{i\in C} \alpha_i \geqslant n,$$
as required.
\end{proof}

\begin{lem}\label{NTIJ} Let $G$ be an odd cycle and $\x^{\alb}$ a monomial of $R$. Assume that for every minimal vertex cover of $G$, we have
$$\sum_{i\in C} \alpha_i \geqslant \sum_{i\notin C}\alpha_i +1.$$
Then, $\alpha_i \geqslant 1$ for every $i=1,\ldots,r$.
\end{lem}
\begin{proof} Without loss of generality we prove that $\alpha_1=0$. Suppose that $V(G) = \{1,2,\ldots,2m+1\}$ for some $m\geqslant 1$. By applying the assumption for  the minimal vertex cover
$C_1 = \{1,3,\ldots,2m+1\}$, we have
$$\sum_{i=0}^m \alpha_{2i+1}\geqslant \sum_{i=1}^m \alpha_{2i} + 1,$$
and by applying the assumption for  the minimal vertex cover
$C_2 = \{1,2,4,\ldots, 2m\}$, we have
$$\alpha_1+\sum_{i=1}^m \alpha_{2i} \geqslant  \sum_{i=1}^m \alpha_{2i+1} + 1.$$

By adding two corresponding sides of two inequalities above together, we obtain
$$2\alpha_1 +\sum_{i=2}^{2m+1}\alpha_i \geqslant \sum_{i=2}^{2m+1}\alpha_i +2.$$
Hence, $\alpha_1\geqslant 1$. 
\end{proof}

The following lemma can verify straightforward.

\begin{lem} \label{sys-edge} Let $\alpha_1,\ldots,\alpha_{2m+1} \in \R$ where $m\geqslant 1$. Then, the linear system
$$\begin{cases}
x_1 + x_2 = \alpha_1\\
x_2+x_3=\alpha_2\\
\cdots\\
x_{2m}+x_{2m+1} = \alpha_{2m}\\
x_{2m+1} + x_1 = \alpha_{2m+1},
\end{cases}
$$
has the unique solution as follows. For each $i = 1,\ldots, 2m+1,$
$$
x_i = \frac{1}{2}\sum_{j=0}^{2m}(-1)^j \alpha_{i+j} =  \frac{1}{2}(\alpha_i-\alpha_{i+1}+\cdots+\alpha_{i+2m-2}-\alpha_{i+2m-1} + \alpha_{i+2m}),
$$
where we use the convention that $\alpha_{i+2m+2}=\alpha_i$ for $i=1,\ldots,2m+1$.
\end{lem}

\begin{lem} \label{EX001} Let $G$ be a graph such that $\mi = \overline{I(G)^n} \colon \x^{\alb}$ for some $\alb \in \N^r$ and $n\geqslant 1$. Then,
\begin{enumerate}
\item $|\alb| = 2n-1$.
\item If $G$ is a cycle, then $\alpha_i \geqslant 1$ for all $i$.
\end{enumerate}
\end{lem}
\begin{proof} For simplicity, let $I = I(G)$. For each supporting hyperplane $\left<\a,\x\right> = 1$ of $NP(I)$ parallel with some direction $\e_i$ which means $a_i =0$, since $\x^{\alb}x_i =\x^{\alb+\e_i}\in \overline{I^n}$, we have
$$\left<\a,\alb\right> = \left<\a,\alb+\e_i\right> \geqslant n.$$

Since $\x^{\alb}\notin \overline{I^n}$, by Lemma \ref{NPH} we must have $\left<\b,\alb\right> < n$ where $\left<\b,\x\right> = 1$ is a supporting hyperplane of $NP(I)$ which is not parallel to any direction $\e_i$ for $i=1,\ldots,r$. Since each generators of $I(G)$ is of degree $2$, by Lemma \ref{NPH} the separating hyperplane $\left<\b,\x\right> = 1$ is of the form
$$\frac{1}{2}(x_1+\cdots +x_r) = 1.$$

It follows that $|\alb| = \alpha_1+\cdots+\alpha_2 \leqslant 2n-1$. On the other hand, since $\x^\alpha x_r = \x^{\alb+\e_r} \in \overline{I^n}$ we have
$$\frac{1}{2}(\alpha_1+\cdots +\alpha_{r-1}+(\alpha_r+1)) \geqslant n,$$
or equivalently, $\alpha_1+\cdots+\alpha_r \geqslant 2n-1$. Thus,
\begin{equation}\label{cycle-eq}
\alpha_1+\cdots+\alpha_r = 2n-1,
\end{equation}
and $(1)$ follows.

Now assume $G$ be a cycle. Then, it must be an odd cycle by \cite[Proposition 3.3]{MMV}. Let $C$ be a minimal vertex of $G$. Then, by Lemma \ref{vertex-cover}
\begin{equation}\label{cycle}
\sum_{i\in C} \alpha_i \geqslant n.
\end{equation}
Together with Formula (\ref{cycle-eq}) we have
$$\sum_{i\notin C}\alpha_i + 1 = |\alpha| - \sum_{i\in C}\alpha_i +1 \leqslant 2n-1-n + 1 = n \leqslant \sum_{i\in C}\alpha_i.$$
Therefore, $\alpha_i \geqslant 1$ for all $i$ by Lemma \ref{NTIJ}, and the lemma follows.
\end{proof}

\begin{lem}\label{EX002} Let $G$ be a graph consisting of disjoint cycles $C_1,\ldots, C_p$, where $p\geqslant 2$. Assume that $\mi = \overline{I(G)^n}\colon \x^{\alb}$ for some $\alb\in\N^r$ and $n\geqslant 1$. We have:
\begin{enumerate}
\item $\alpha_i \geqslant 1$ for all $i$.
\item If $\sum_{i\in V(C_j)} \alpha_i =2k$ for some $j$ and $k \geqslant 1$, then $\alpha_i \geqslant 2$ for all $i\in V(C_j)$.
\end{enumerate}
\end{lem}
\begin{proof} $(1)$ Fix arbitrary index $i$, we need to prove that $\alpha_i \geqslant 1$. Without loss of generality, we may assume that $i\in V(C_1)$. Let $C$ be a minimal cover of $C_1$. We first claim that
\begin{equation}\label{first-claim}
\sum_{i\in C_1} \alpha_i \geqslant \sum_{i\in V(C_1)\setminus C}\alpha_i + 1.
\end{equation}

Indeed, by Lemma \ref{NTIJ}, every $\btb\in \N^r$ with $\x^{\btb}\in I(C_1)$ is a solution of
$$\sum_{i\in C} x_i \geqslant 1.$$

Next,  every $\btb\in \N^r$ with $\x^{\btb}\in I(H)$ also a solution of
$$\frac{1}{2}(x_{2m+2}+\cdots +x_r) \geqslant 1$$
since if it holds true whenever $\x^{\alb}$ is a generator of $I(G)$.

From two inequalities above, we conclude that every every $\btb\in \N^r$ with $\x^{\btb}\in I(G)$ is a solution of 
$$\sum_{i\in C} x_i  + \frac{1}{2}(x_{2m+2}+\cdots +x_r) \geqslant 1.$$
It follows that all $\btb\in \N^r$ with $\x^{\btb}\in \overline{I(G)^n}$ is a solution of 
\begin{equation}\label{EXE01}
\sum_{i\in C} x_i  + \frac{1}{2}(x_{2m+2}+\cdots +x_r) \geqslant n.
\end{equation}

Now let $i\in V(G_1)$ such that $i\notin C$. Since $x_i\x^{\alb} \in \overline{I^n}$, by Inequality $(\ref{EXE01})$ we imply that
$$\sum_{i\in C} \alpha_i + \frac{1}{2}\sum_{i \geqslant 2m+2}\alpha_i \geqslant n.$$
Hence,
$$\sum_{i\in C} \alpha_i - \sum_{i\in V(G_1)\setminus C} \alpha_i +|\alb| \geqslant 2n.$$
On the other hand, by Lemma \ref{EX001} we have $|\alb| = 2n-1$, so
$$\sum_{i\in C} \alpha_i - \sum_{i\in V(G_1)\setminus C} \alpha_i \geqslant 1,$$
and the claim follows.

Now by Lemma \ref{NTIJ} and this claim we obtain $\alpha_i \geqslant 1$ for $i\in V(C_1)$, and $(1)$ holds.

\medskip

$(2)$ We may assume that $\sum_{i\in V(C_1)} \alpha_i =2k$. Let $H = C_2\cup\cdots\cup C_p$. For simplicity, we set $I = I(C_1)$,  $J = I(H)$ and  $E(G(J)) = \{\f_1,\ldots,f_q\}$ for $q\geqslant 1$. Note  that $|\f_j| = 2$ for $j=1,\ldots,q$.
Note also that $$E(G(I)) = \{\e_1+\e_2,\e_2+\e_3,\ldots,\e_{2m}+\e_{2m+1}, \e_{2m+1},\e_1\}.$$

\medskip

Since $\x^{\alb} x_{r} \in \overline{I(G)^n}$, we have
$$\alb+\e_r =\sum_{i=1}^{2m+1} \lambda_i (\e_i+\e_{i+1}) + \sum_{i=1}^{q} \gamma_j \f_j + \v.$$
where $\lambda_i\geqslant 0, \gamma_j\geqslant 0$ and $\v\in \R_+^r$ such that 
$$\sum_{i=1}^{2m+1} \lambda_i + \sum_{i=1}^{q} \gamma_j = n.$$
Since $|\alb+\e_r| = 2n$, it follows that $|\v| = 0$, whence $\v = 0$. Therefore,
$$\alb + \e_r =\sum_{i=1}^{2m+1} \lambda_i (\e_i+\e_{i+1}) + \sum_{i=1}^{q} \gamma_j \f_j.$$

Note also that $r > 2m+1$. By comparing the first $2m+1$ coordinates of two vectors in $\R^r$ in both sides we get the following linear system
$$
\begin{cases}
\alpha_1 = \lambda_1 + \lambda_{2m+1}\\
\alpha_2 = \lambda_1 + \lambda_2\\
\cdots\\
\alpha_{2m} = \lambda_{2m-1}+\lambda_{2m}\\
\alpha_{2m+1} = \lambda_{2m}+\lambda_{2m+1}.
\end{cases}
$$
By Lemma \ref{sys-edge} we have 
\begin{equation}\label{abc01}
\lambda_{1} = \frac{1}{2}(\alpha_2-\alpha_3+\cdots+\alpha_{2m}-\alpha_{2m+1}+\alpha_1),
\end{equation}
and
\begin{equation}\label{abc02}
\lambda_{2m+1} = \frac{1}{2}(\alpha_1-\alpha_2+\cdots+\alpha_{2m-1}-\alpha_{2m}+\alpha_{2m+1}).
\end{equation}

On the other hand, since $C_1 = \{2,4,\ldots,2m,1\}$ and $C_2 = \{1,3,\ldots, 2m+1\}$ are vertex covers of $G$, together with Lemma \ref{NTIJ} and Formulas $(\ref{abc01})$ and $(\ref{abc02})$ we deduce that $\lambda_1\geqslant 1/2$ and $\lambda_{2m+1}\geqslant 1/2$.

Now we prove that $\alpha_1 \geqslant 2$. Assume on the contrary that $\alpha_1 < 2$. Together with Part $(1)$, this fact forces $\alpha_1 = 1$. Since $\alpha_1 = \lambda_1 + \lambda_{2m+1}$ and $\lambda_1\geqslant 1/2$ and $\lambda_{2m+1} \geqslant 1/2$, we have $\lambda_1=1/2$. Together with Formula $(\ref{abc01})$, it gives
$$\alpha_2-\alpha_3+\cdots+\alpha_{2m}-\alpha_{2m+1} = 2\lambda_1 -\alpha_1 = 0,$$
or
$$\alpha_2+\cdots+\alpha_{2m} = \alpha_3 +\cdots+\alpha_{2m+1}.$$
It follows that
$$\sum_{i\in V(C_1)} \alpha_i = 2(\alpha_3+\cdots+\alpha_{2m+1}) + \alpha_1 = 2(\alpha_3+\cdots+\alpha_{2m+1}) + 1,$$
which is an odd number. This contradict the assumption that $\sum_{i\in V(C_1)} \alpha_i = 2k$, and so $\alpha_1\geqslant 2$. The proof of the lemma is complete.
\end{proof}

The following lemma is the same as \cite[Lemma 2.10]{MO} for ordinal powers of edge ideals

\begin{lem}\label{leaf} Let $G$ be a graph and $v$ is an its leaf. Assume that $u$ is the unique neighbor of $u$ in $G$. Then, for all $n\geqslant 2$, we have $\overline{I(G)^n} \colon x_ux_v = \overline{I(G)^{n-1}}$.
\end{lem}
\begin{proof} Let $f$ be a monomial of $R$ such that $fx_ux_v \in \overline{I(G)^n}$. By Equation (\ref{EN1}), $(fx_ux_v)^m \in I(G)^{nm}$ for some $m\geqslant 1$. It follows that
$f^m \in I(G)^{nm} \colon (x_ux_v)^m$. On the other hand, by \cite[Lemma 2.10]{MO} we have  $I(G)^{nm} \colon (x_ux_v)^m = I(G)^{m(n-1)}$. It follows that
$f^m \in I(G)^{m(n-1)}$, and therefore $f \in \overline{I(G)^{n-1}}$. Hence, $\overline{I(G)^n} \colon x_ux_v \subseteq \overline{I(G)^{n-1}}$.

The reverse inclusion follows from Lemma \ref{I1} since $x_ux_v \in I(G)$, and the proof is complete.
\end{proof}

The following lemma shows that the bound in Lemma \ref{A3} is sharp for pseudoforests.

\begin{lem}\label{main-as-one} If $G$ is a pseudoforest, then $\mi \in \ass(R/\overline{I(G)^n})$ if and only if $n\geqslant  n_1(G)$.
\end{lem}
\begin{proof} The if part follows from Lemma \ref{A3}, so we prove the only if part. 

Suppose that $G_1,\ldots, G_p$ are connected components of $G$. Let $2m_i-1$ be the length of the odd cycle of $G_i$ for all $i$; and $2m-1$ the minimum length of odd cycles of $G$.

We will prove by induction on $\upsilon(G)$ that if $\mi \in \ass(R/\overline{I(G)^n})$, then $n\geqslant n_1(G)$. By Lemma \ref{HIO} we may assume that $n$ is minimal.

First we consider the base case is when $G$ has no leaves, i.e. each $G_i$ is a cycle.  We consider two possible subcases:

{\it Subcase $1$}: $p = 2s+1$ for some $s\geqslant 0$. By Lemma \ref{EX002} we have $\alpha_i \geqslant 1$ for all $i$. Together with Lemma \ref{EX001} we have
$$2n-1 = |\alb| \geqslant r,$$
and so 
\begin{align*}
n & \geqslant \frac{1}{2}(r+1) = \frac{1}{2}((2m_1-1)+\cdots+(2m_{2s+1}-1)+1)\\
&=m_1+\cdots + m_{2s+1} - s = n_1(G),
\end{align*}
and the lemma holds in this case.
\medskip

{\it Subcase $2$}: $p = 2s$ for some $s\geqslant 1$. By Lemma \ref{EX001} we have 
$$|\alb|  = \sum_{j=1}^{2s} \left(\sum_{i\in V(C_j)} \alpha_i\right) = 2n-1.$$
Therefore, $\sum_{i\in V(C_j)} \alpha_i$ is even for some $i$. Without loss of generality, we may assume that $\sum_{i\in V(C_1)} \alpha_i$ even. 

By Lemma \ref{EX002} we have
$$\alpha_i \geqslant 2 \text{ for all } i\in V(C_1),\text{ and } \alpha_i \geqslant \text{ for all } i\notin V(C_1).$$
It follows that
\begin{align*}
2n-1 &= |\alb| =\sum_{i\in V(C_1)} \alpha_i+ \sum_{j=2}^{2s} \left(\sum_{i\in V(C_j)} \alpha_i\right)\geqslant 2(2m_1-1) +  \sum_{j=2}^{2s} (2m_j-1)\\
&= 2\sum_{j=1}^{2s}m_i - 2s+ 2m_1 - 1.
\end{align*}
Hence,
$$n \geqslant \sum_{j=1}^{2s}m_i - s + m_1\geqslant \sum_{j=1}^{2s}m_i - s + m = n_1(G),$$
and hence  the lemma holds in this case.

\medskip

We next consider the case $G$ has leaves. We start by proving  four following  claims:
\medskip

{\it Claim $1$:} If a vertex $v$ of $G$ is adjacent to a leaf, then $\alpha_v\geqslant 1$. Indeed, let $u$ be a leaf of $G$ adjacent to $v$. Since $x_u \x^{\alb} \in \overline{I(G)^n}$, by Equation (\ref{EN1}) one has $x_u^m \x^{m\alb} \in I(G)^{nm}$ for some $m\geqslant 1$. Since $u$ is a leaf of $G$, we deduce that $x_ux_v$ divides $x_v^m \x^{m\alb}$, and so $\alpha_v\geqslant 1$, as claimed.  
\medskip

{\it Claim $2$}: $\alpha_v = 0$ for every leaf $v$. Indeed, assume on contrary that $\alpha_v \geqslant 1$. Let $u$ be a neighbor of $v$. Then, $\alpha_u\geqslant 1$ by Claim $1$. Hence, we can write $\x^{\alb} = (x_ux_v)f$ for some monomial $f$. By Lemma \ref{leaf} we have
$$\mi = \overline{I(G)^{n}} \colon  (x_u x_v)f =  (\overline{I(G)^{n}} \colon  x_u x_v) : f =  \overline{I(G)^{n-1}} \colon  f,$$
so $\mi \in \ass(R/\overline{I(G)^{n-1}})$, which contradicts the minimality of $n$. Therefore, $\alpha_v = 0$, as claimed.

\medskip

{\it Claim $3$}: If $u$ is a non-leaf vertex such that $N_G(u)$ has only one a non-leaf vertex $v$, then $\alpha_v \geqslant 1$.

Indeed, assume on the contrary that $\alpha_v = 0$. Together with Claim $2$, it forces $\alpha_i = 0$ for all $i\in N_G(u)$. Since $x_u x^{\alb} \in \overline{I(G)^n}$ but $x^{\alb} \notin \overline{I(G)^n}$. It follows that there is $m\geqslant 1$ such that
$$x_u^m \x^{m\alb} \in I(G)^{nm} \text{ but } \x^{m\alb} \notin I(G)^{nm}.$$
It follows that 
\begin{equation}\label{BL}
x_u^m \x^{m\alb} = gf_1\cdots f_{nm},
\end{equation}
where $f_1,\ldots,f_{nm}$ are quadratic monomials of $I(G)$ and $g$ is a monomial of $R$. If $x_u \mid f_j$ for some $j$, then $f_j = \{u,i\}$ for some $i \in N_G(u)$. In this case, by Formula $(\ref{BL})$, we must have $\alpha_i \ne 0$, a contradiction. Hence, $x_u \nmid f_j$ for all $j$. Together with Formula $(\ref{BL})$, it yields $x_u^m \mid g$. Thus, $f_1\cdots f_{nm} \mid \x^{\alb}$, and thus $\x^{\alb} \in \overline{I(G)^n}$, a contradiction. Therefore, we must have $x_v \geqslant 1$, as claimed.

\medskip

The following claim can check straightforward by using Claim $2$.
\medskip

{\it Claim $4$}: If $v$ is a leaf and $H = G\setminus v$, then $(x_i\mid i\in V(H)) = \overline{I(H)^n} \colon \x^{\alb}$. 

\medskip 

We now continue to prove the lemma. Let $v$ be a leaf of $G$ and $u$ is the neighbor of $v$. Let $H = G \setminus v$. We consider two subcases:
\medskip

{\it Subcase $1$}: $u$ is not a leaf of $H$. Then, 
$$\upsilon(G) = \upsilon(H)+1, \ \varepsilon_0(G) = \varepsilon_0(H) +1,$$
and hence $n_1(G) = n_1(H)$. Since $\upsilon(H) < \upsilon(H)$, together Claim $4$ and the hypothesis induction we have
$$n \geqslant n_1(H) = n_1(G),$$
and the lemma holds.

{\it Subcase $2$}: $u$ is a leaf of $H$. Then, 
$$\upsilon(G) = \upsilon(H)+1, \ \varepsilon_0(G) = \varepsilon_0(H),$$
and hence $n_1(G) = n_1(H)+1$.

Let $i$ be a neighbor of $u$ in $H$. Observe that $i$ is the only non-leaf vertex in $N_G(u)$, so that $\alpha_i \geqslant 1$ by Claim $3$. It follows that $\x^{\alb} = (x_ux_i)f$ for some monomial $f$ of $R$. By Claim $4$ and Lemma \ref{leaf} we have
$$(x_j \mid j \in V(H)) = \overline{I(H)^n} \colon \x^{\alb} = (\overline{I(H)^n} \colon x_ux_i)\colon f = \overline{I(H)^{n-1}} \colon f.$$

Since $\upsilon(H) < \upsilon(G)$, by the induction hypothesis we have $n-1\geqslant n_1(H)$. Hence, $n \geqslant n_1(H)+1=n_1(G)$. The proof of the lemma is complete.
\end{proof}

We are now prove the bound in Theorem \ref{T3} is sharp for pseudoforests.

\begin{thm} \label{T6} If $G$ is a pseudoforest, then $\overline{\astab}(I(G)) = \phi_0(G)$.
\end{thm}
\begin{proof}  By Theorem \ref{T3}, it suffices to show that $\overline{\astab}(I(G)) \geqslant \phi_0(G)$.  By Lemma \ref{L03} and Proposition \ref{T1} we may assume that all connected components of $G$ are non-bipartite. 

Suppose that $G_1,\ldots, G_p$ be connected components of $G$ such that $$\phi_0(G) = n_0(G_1 \cup \cdots \cup G_s), \text{ for } 1 \leqslant s \leqslant p.$$

Note that if every connected bipartite component of a graph $H$ is unicyclic, then $n_0(H) = n_1(H)$. Hence,
$$\phi_0(G) = n_1(G_1 \cup \cdots \cup G_s).$$

Let $n = \overline{\astab}(I(G))$. Since $\mi \in \ass(R/\overline{I(G)^m}$ for some $m\geqslant 1$ by Lemma \ref{main-as-one}, and hence
\begin{equation}\label{T100}
\mi \in \ass(R/\overline{I(G)^n}.
\end{equation}

If  $s = p$,  together with Lemma \ref{main-as-one}, Formula (\ref{T100}) implies that $n \geqslant n_1(G) = \phi_0(G)$, and so the theorem holds in this case. 

Assume that $1\leqslant s < p$. Let $G' = G_1 \cup \cdots\cup G_s$ and $G'' = G_{s+1} \cup\cdots \cup G_p$. Let $S$ be a maximal independent set of $G''$. Then, $C = V(G'')\setminus S$ is a minimal cover of $G''$. Moreover, $I(G)_S = \mi_C + I(G')$, where $\mi_C = (x_i \mid i\in C)$. 

Let $\nn = (x_i \mid i \in V(G'))$. By Lemma \ref{main-as-one}, we have $\nn \in \ass(R/\overline{I(G')^m})$ for some $m\geqslant 1$. By Proposition \ref{T2} we have $\p = \mi_C + \nn \in \ass(R/\overline{I(G)_S^m})$. Together with Lemma \ref{DI}, it yields
$\p \in \ass(R/\overline{I(G)^m})$. 

On the other hand, since $n = \overline{\astab}(I(G))$, by Lemma \ref{HIO} we get $\p \in  \ass(R/\overline{I(G)^n})$. Since $\p_S = \p$, we have
$$\p \in \ass(R/(\overline{I(G)^n})_S) = \ass(R/\overline{I(G)_S^n}) = \ass(R/\overline{\mi_C + I(G')^n}).$$
Together with Proposition $(\ref{T2})$, it forces $\nn \in \ass(R/\overline{I(G')^n})$. By Lemma \ref{main-as-one}, we obtain $n\geqslant n_1(G') = \phi_0(G)$, and the theorem follows.\end{proof}

Finally, we  prove the bound in Theorem \ref{T4} is sharp for pseudoforests..
\begin{thm} \label{T8} If $G$ is a pseudoforest with no cycles of length $4$, then $\overline{\dstab}(I(G)) = \phi_1(G)$.
\end{thm}
\begin{proof} Let $G'$ be a subgraph of $G$ consisting of all bipartite connected components of $G$ and $G''$ consisting of all nonbipartite connected components of $G$. Then, we have
$$\overline{\dstab}(I(G)) = \dstab(I(G')) + \overline{\dstab}(I(G''))-1.$$
Note that 
$$\overline{\dstab}(I(G'')) = \min\{n\geqslant 1\mid \mi \in \ass(R/\overline{I(G)^n}\} = n_1(G).$$ 
Hence, the theorem follows from \cite[Theorem 5.1]{Tr} and Lemma \ref{main-as-one}.
\end{proof}

\subsection*{Acknowledgment}  The second author would like to thank Vietnam Institute for Advanced Study in Mathematics for the hospitality during his visit in 2020, when he started to work on this paper.

\end{document}